\documentclass[12pt]{amsart}
\usepackage{amssymb, latexsym,framed}
\usepackage{csquotes}
\usepackage{amsthm}
\usepackage{amsmath,amscd}
\usepackage{amsfonts}
\usepackage[dvips]{graphicx}
\usepackage{amsthm}
\usepackage{amsmath,amscd}
\usepackage{amsfonts}
\usepackage[dvips]{graphicx}
\usepackage{color}
\numberwithin{equation}{section} 
\newcommand{\bean}{\begin{eqnarray*}}
\newcommand{\eean}{\end{eqnarray*}}
\newcommand{\be}{\begin{equation}}

\newcommand{\ee}{\end{equation}}
\newcommand{\bd}{\begin{displaymath}}
\newcommand{\ed}{\end{displaymath}}

\newcommand{\beq}{\begin{equation}}
\newcommand{\eeq}{\end{equation}}
\newcommand{\bea}{\begin{eqnarray}}
\newcommand{\eea}{\end{eqnarray}}

\newcommand{\abs}[1]{\left\vert{#1}\right\vert}
\newcommand{\proj}{\mathrm{Proj}}

\newcommand{\R}{\mathbb{R}}

\newcommand{\e}{\varepsilon}
\newcommand{\norm}[1]{\left\Vert#1\right\Vert}

\DeclareMathOperator{\dist}{dist}

\DeclareMathOperator{\sing}{Sing}

\newtheorem{thm}{Theorem}
\newtheorem{prop}[thm]{Proposition}
\newtheorem{lem}[thm]{Lemma}
\theoremstyle{definition}
\newtheorem{rem}[thm]{Remark}

\def\({\left(}
\def\){\right)}
\def\1{\mathds{1}}

\def\l|{\left|}

\def\p{\partial}

\def\r|{\right|}

\def\XXint#1#2#3{{\setbox0=\hbox{$#1{#2#3}{\int}$ }
\vcenter{\hbox{$#2#3$ }}\kern-.6\wd0}}

\begin{document}

\begin{abstract}
The study of singular perturbations of the Dirichlet energy is at the core of the phenomenological-description paradigm in soft condensed matter. Being able to pass to the limit plays a crucial role in the understanding of the geometric-driven profile of ground states. In this work we study, under very general assumptions, the convergence of minimizers  towards harmonic maps. We show that the convergence is locally uniform up to the boundary, away from the lower dimensional singular set. Our results generalize related findings, most notably in the theory of liquid-crystals, to all dimensions $n\geq 3$, and to general nonlinearities. Our proof follows a well-known scheme, relying on  small energy estimate and monotonicity formula. It departs substantially from previous studies in the treatment of the small energy estimate at the boundary, since we do not rely on the specific form of the potential. In particular this extends existing results in 3-dimensional settings. In higher dimensions we also deal with additional difficulties concerning the boundary monotonicity formula.

\end{abstract}

\title{On the convergence of minimizers of singular perturbation functionals}
\date{}
\author{Andres Contreras \and Xavier Lamy \and R{\'e}my Rodiac}

\address[A.Contreras]{Department of Mathematical Sciences, New Mexico State University, Las Cruces,
New Mexico, USA}
\email{acontre@nmsu.edu}

\address[X.Lamy]{Max Planck Institute for Mathematics in the Sciences, Leipzig, Germany
}
\email{xlamy@mis.mpg.de}
\address[R.Rodiac]{Facultad de Matematic\'as, Pontificia Universidad Cat\'olica de Chile, Vicu\~na Mackenna 4860, Macul, Santiago Chile}
\email{remy.rodiac@mat.uc.cl}

\maketitle

\section{Introduction}\label{s:intro}

In this article, our main interest is the asymptotic behavior of minimizers $(u_\e)$ of the Ginzburg-Landau type energy functionals

%

\begin{equation}\label{E}
E_\e(u)=\frac 12 \int_\Omega\abs{\nabla u}^2 +\frac 1 {\e^2} \int_\Omega f(u),\qquad u\in H^1(\Omega;\R^k),
\end{equation}
subject to fixed boundary conditions
\begin{equation}\label{ub}
u_{\lfloor\partial\Omega}=u_b\in C^2(\partial\Omega;\mathcal N),
\end{equation}
where $\Omega\subseteq \R^n,$ $n\geq 3$ and $f:\R^k\to [0,\infty)$ is a smooth potential such that its vacuum

 \begin{equation}\label{N}
\mathcal N:=\lbrace f=0\rbrace ,\quad\text{ is a smooth compact submanifold of }\R^k.
\end{equation}

%
{
The functional \eqref{E} can be seen as a relaxation of the Dirichlet energy $\int\abs{\nabla u}^2$ for $\mathcal N$-valued maps.
Energy functionals of the form \eqref{E} are} very common in the theory of phase transitions and instances of it are the Allen-Cahn functional, the Ginzburg-Landau energy and the Landau-de Gennes model, to name a few.

Our goal is to establish a stronger compactness of $(u_\e)$ than the one readily available by classical soft arguments (see \eqref{soft.comp} below); we show the existence of a subsequential $H^1$-limit, a harmonic map satisfying \eqref{ub}, such that the convergence is actually uniform away from the singular set of the limiting $\mathcal{N}$-valued map. Our main theorem proves that this is a robust phenomenon that does not depend strongly on the particular potential $f,$ in fact the only assumptions we make are

\begin{itemize}
\item[(f1)]
 There exists $R>0$ satisfying
\begin{equation}\label{radialgrowthpotential}
\abs{z}{\geq} R \quad\Longrightarrow \nabla f(z)\cdot z \geq 0.
\end{equation}

\item[(f2)]
Minimizers $u=u_\e$ of $E_\e$ solve\footnote{ Under rather natural growth conditions on the potential $f ,$  hypothesis $(f_2)$ is satisfied and it is therefore not a restrictive requirement (see \textit{e.g.} \cite{dacorogna}).} the semilinear elliptic system
\begin{equation}\label{pde}
\Delta u =\frac{1}{\e^2}\nabla f(u)\quad\text{in }\mathcal D'(\Omega).
\end{equation}
\item[(f3)] Generic assumption: $f$ vanishes non-degenerately on $\mathcal N$, that is
\begin{equation}\label{nondegen}
\nabla^2 f(x) v\cdot v > 0\quad \text{for } x\in \mathcal N\text{ and }v\in (T_x\mathcal N)^\perp\setminus \lbrace 0\rbrace.
\end{equation}
\end{itemize}
{Here $T_x\mathcal N$ denotes the tangent space to $\mathcal N$ at $x$ and $(T_x\mathcal N)^\perp$ its orthogonal complement in $\R^k$.}

\begin{rem}\label{rem:usmooth}
The assumption \eqref{radialgrowthpotential} on $f$ ensures that distributional solutions of \eqref{pde} that belong to $H^1$ satisfy a uniform bound \cite[Lemma~8.3]{lamy14}
\begin{equation}\label{unifbound}
\norm{u}_{L^\infty}\leq R + \norm{u_b}_{L^\infty}.
\end{equation}
Therefore by  elliptic regularity such $u$ is smooth.
\end{rem}
{
\begin{rem}
Relevant examples of potentials satisfying (f1)--(f3) include the Ginzburg-Landau potentials $f: \R^k \rightarrow \R, z\mapsto (1-|z|^2)^2$ with $\mathcal{N}=\mathbb{S}^{k-1}$, and the Landau-de Gennes potential (see discussion after the statement of Theorem~\ref{thm:main}).
\end{rem}
}
As $\e\to 0$, any minimizing family $(u_\e)$
admits a subsequence converging strongly in $H^1$ to a map

\begin{equation}\label{soft.comp}
u_\star\in H^1(\Omega;\mathcal N),
\end{equation}
which minimizes the Dirichlet energy $\int\abs{\nabla u}^2$ among $\mathcal N$-valued maps, subject to the boundary conditions \eqref{ub}. This can be checked as in \cite[Lemma~3]{majumdarzarnescu10}. It is well known that $u_\star$ is not smooth in general, so that we cannot expect the convergence of $u_\e$ towards $u_\star$ to be uniform in $\Omega$. On the other hand such uniform convergence might be expected away from the singular set $\sing (u_\star)$, which is a compact subset of $\Omega$ of Hausdorff dimension at most $(n-3)$ \cite{SUregularity,SUboundaryregularity}. Our main result states that this is indeed the case:
\begin{thm}\label{thm:main}
Assume \emph{(f1)--(f3)} hold. If a subsequence of minimizers $(u_\e)$ of $E_\e$ subject to \eqref{ub} converges strongly in $H^1$, then it holds in fact that
\begin{equation*}
u_\e\longrightarrow u_\star\quad\text{locally uniformly in }\overline\Omega\setminus \sing(u_\star).
\end{equation*}
\end{thm}

One of the main motivations for studying this problem comes from questions arising in the Landau-de Gennes theory of liquid crystals, where $n=3$ and $f$ is a particular potential defined on the space of symmetric and traceless $3\times 3$ matrices. The simplified Landau-de Gennes functional is given by
\[\mathcal{F}_{LG}[Q]=\int_{\Omega}\frac{L}{2}\abs{\nabla Q}^2(x)+f_B (Q(x))dx ,\]
where $L$ is the so-called {\it elastic constant} and the transition term is $\abs{\nabla Q}^2=\sum_{i,j,k=1}^3 Q_{ij,k}Q_{ij,k}.$ The potential takes the explicit form
\begin{equation*}
f_B(Q)=\frac{\alpha^2(T-T^{*})}{2}tr(Q^2)-\frac{b^2}{3}tr(Q^3)+\frac{c^2}{4}(tr Q^2)^2,
\end{equation*}
and is the simplest example of a multi-well potential.   In this way, the Landau-de Gennes energy corresponds to $E_\e$ for the particular choice of potential $f_B$, in the vanishing elastic constant regime $L\sim 0$.
{It can be checked that for $T\leq T^*$ the potential $f_B$ satisfies {(f1)--(f3)} for $\mathcal{N}=\{s_*(n\otimes n -\frac{1}{3}I)\colon n \in \mathbb{S}^2\}$ and some $s_*>0$.}

In the case of the Landau-de Gennes energy, Theorem~\ref{thm:main} has been proved by Majumdar and Zarnescu for the interior convergence \cite{majumdarzarnescu10} and by Nguyen and Zarnescu for the convergence up to the boundary \cite{nguyenzarnescu13} {(see also \cite{Canev} for sequences of unbounded energy)}. These works were building on methods developed for the Ginzburg-Landau energy \cite{BBH0,BBH1}. However, the Ginzburg-Landau and Landau-de Gennes models are only two in a family of increasingly refined and complex physical theories. It is then natural to ask to what extent this uniform convergence depends on the particular model and how sensitive it is to the potential at hand. In this respect our objective is to develop an approach that could potentially encompass all such models.

  Our contribution generalizes the results in \cite{majumdarzarnescu10, nguyenzarnescu13} to general potentials and arbitrary dimension $n\geq 3$. For the interior convergence the techniques adapt  without great difficulties. Regarding the boundary convergence however, the arguments in \cite{nguyenzarnescu13} are really specific to $n=3$ and the particular form of $f$. Let us be more specific and describe the general strategy of the proof. It relies on two main ingredients:
\begin{itemize}
\item a small energy estimate which states that, in a ball where the (appropriately rescaled) energy is small enough at all small scales, $\nabla u_\e$ is uniformly bounded;
\item and a monotonicity formula which allows to show that the energy is small at all small scales, provided it is small at one fixed scale.
\end{itemize}
The boundary monotonicity formula in \cite{majumdarzarnescu10} is derived under the assumption that $n=3$, and it is not clear whether such formula holds for general $n\geq 3$. Here we obtain a weaker version of it, which turns out to be enough for our purposes. On the other hand, the proof of the small energy estimate in \cite{nguyenzarnescu13} relies quite strongly on the particular structure of the potential. We provide a simpler proof that uses only the assumption of nondegeneracy \eqref{nondegen}. As in \cite{nguyenzarnescu13} the main ingredient is a Bochner-type identity, an elliptic equation satisfied by $\abs{\nabla u}^2$. To make use of it, one first needs some estimates on $\nabla u$ at the boundary, and we remark here that they can be obtained quite directly by computations similar to those in \cite{chenlin93}.

 In connection with the physical motivation of the problem, it would be interesting to replace the Dirichlet boundary conditions by the so-called weak anchoring conditions, which are enforced by adding an anchoring term to the energy functional. Such boundary conditions are more physically relevant, for instance in the study of nematic colloids \cite{alamabronsardlamy,alamabronsardlamy16}. In the case of weak anchoring, the limit $u_\star$ also enjoys some partial regularity \cite{contreraslamyrodiac15}. However the strategy detailed above for obtaining uniform convergence near the boundary seems much harder to implement, since it is not clear whether an equivalent of Lemma~\ref{lem:bdryestim} below would hold. In \cite{futurework} we use different methods to tackle this problem.

The article is organized as follows. In Section~\ref{s:monot} we prove the boundary monotonicity formula, in Section~\ref{s:estim} we prove the small energy estimate and we conclude in Section~\ref{s:proofmain} with the proof of Theorem~\ref{thm:main}.

\section{Monotonicity formula}\label{s:monot}

In this section and in the rest of the article we denote by $e_\e(u)$ the energy density
\begin{equation*}
e_\e(u) = \frac 12 \abs{\nabla u}^2 +\frac{1}{\e^2}f(u),
\end{equation*}
and prove the following boundary monotonicity formula:
\begin{prop}\label{prop:monot}
There exists a constant $K\geq 0$ depending only on $\Omega$ and $u_b$, such that for all $x_0\in\overline\Omega$ and any $\e\in (0,1)$ the function
\begin{equation*}
\psi(\rho):=2K\rho + \rho^{2-n}\int_{\Omega\cap B_\rho(x_0)} e_\e(u_\e),
\end{equation*}
satisfies
\begin{equation}\label{monot}
\frac{d}{d\rho}\psi(\rho)\geq K(1-\psi(2\rho)),
\end{equation}
for all $\rho\in (0,1)$.
\end{prop}
\begin{proof}
To simplify notation we drop the explicit dependence on $\e$ and write $u$ for a minimizer of \eqref{E} under the boundary condition \eqref{ub}. We use coordinates in which $x_0=0$ and let $\varphi(\rho)$ denote the renormalized energy
\begin{equation}\label{renormen}
\varphi(\rho)=\rho^{2-n}\int_{\Omega\cap B_\rho}e_\e(u),
\end{equation}
so that $\psi(\rho)=2K\rho+\varphi(\rho)$.

Before proceeding with the proof, let us recall that smooth  solutions to the Euler-Lagrange equations {\eqref{pde}} satisfy that their associated stress-energy tensor is divergence free
\begin{equation}\label{statio}
\partial_\ell T_{\ell j} =0,\qquad T_{\ell j} :=\partial_\ell u\cdot\partial_j u -\left(\frac {1}2 \abs{\nabla u}^2+ \frac{1}{\e^2}f(u)\right)\delta_{j\ell}.
\end{equation}
As usual, the monotonicity formula follows from \eqref{statio}.

The beginning of the proof (until \eqref{monot2} below) is similar to \cite[Lemma~9]{majumdarzarnescu10}.
Multiplying \eqref{statio} by $x_j$ and integrating by parts in $\Omega\cap B_\rho$ yields
\begin{align*}
\frac{d\varphi}{d\rho} & = \frac{2}{\e^2}\rho^{1-n}\int_{\Omega\cap B_\rho}  f(u)
+{\rho^{-n}\int_{\Omega\cap\partial B_\rho}\abs{(x\cdot\nabla) u}^2} \\
&\quad +\rho^{1-n}\int_{\partial\Omega\cap B_\rho} (x\cdot\nabla)u\cdot \frac{\partial u}{\partial\nu} - \rho^{1-n}\int_{\partial\Omega\cap B_\rho} (x\cdot\nu) \, e_\e(u)\\
& \geq  \rho^{1-n}\int_{\partial\Omega\cap B_\rho} (x\cdot\nabla)u\cdot \frac{\partial u}{\partial\nu} - \rho^{1-n}\int_{\partial\Omega\cap B_\rho} (x\cdot\nu) \, e_\e(u)\\
& = \rho^{1-n}\int_{\partial\Omega\cap B_\rho} (x\cdot\nabla)u\cdot \frac{\partial u}{\partial\nu} - \frac 12 \rho^{1-n}\int_{\partial\Omega\cap B_\rho} (x\cdot\nu) \, \abs{\nabla u}^2.
\end{align*}
Here $\nu=\nu(x)$ denotes the exterior unit normal to $\partial\Omega$ at $x$, and for the last equality we have used the fact that $u_{\lfloor \partial\Omega}$ takes values into $\mathcal N$. Introducing the {(non unit) tangential } vector  $\tau(x):=x-(x\cdot \nu)\nu$ and noticing that it holds
\begin{equation*}
(x\cdot\nabla)u\cdot \frac{\partial u}{\partial\nu}=(\tau\cdot\nabla)u \cdot \frac{\partial u}{\partial\nu} + (x\cdot\nu)\abs{\frac{\partial u}{\partial\nu}}^2,
\end{equation*}

we rewrite the above as
\begin{align*}
\rho^{n-1}\frac{d\varphi}{d\rho} & \geq \int_{\partial\Omega\cap B_\rho} (\tau\cdot\nabla)u\cdot \frac{\partial u}{\partial\nu} + \frac 12 \int_{\partial\Omega\cap B_\rho} (x\cdot\nu ) \abs{\frac{\partial u}{\partial\nu}}^2 \\
&\quad -\frac 12 \int_{\partial\Omega\cap B_\rho} (x\cdot\nu ) \left(\abs{\nabla u}^2-\abs{\frac{\partial u}{\partial\nu}}^2\right).
\end{align*}
Since $\tau$ is tangent to $\partial\Omega$ and $\abs{\tau}\leq\rho$ we have
\begin{equation*}
(\tau\cdot\nabla)u\cdot \frac{\partial u}{\partial\nu} \leq \rho \left(\sup\abs{\nabla_{\partial\Omega} u_b}\right)\abs{\frac{\partial u}{\partial\nu}}
\leq \frac 12 \sup\abs{\nabla_{\partial\Omega} u_b}^2 +\frac 12 \rho^2 \abs{\frac{\partial u}{\partial\nu}}^2,
\end{equation*}
and using also that $\abs{\nabla u}^2-\abs{\frac{\partial u}{\partial\nu}}^2=\abs{\nabla_{\partial\Omega}u}^2$ we deduce
\begin{equation}\label{monot1}
\rho^{n-1}\frac{d\varphi}{d\rho}\geq \frac 12  \int_{\partial\Omega\cap B_\rho} (x\cdot \nu -\rho^2) \abs{\frac{\partial u}{\partial\nu}}^2 - \mathcal H^{n-1}(\partial\Omega\cap B_\rho) \sup\abs{\nabla_{\partial\Omega}u_b}^2 .
\end{equation}
Since $\Omega$ is a smooth bounded domain, there exists a constant $C=C(\Omega)>0$ such that for all $x_0\in\overline\Omega$ it holds
\begin{gather*}
\mathcal H^{n-1}(\partial\Omega\cap B_\rho(x_0))\leq C\,\rho^{n-1},\quad\text{and}\\
(x-x_0)\cdot\nu(x)\geq - C\,\abs{x-x_0}^2\quad\text{for }x\in\partial\Omega.
\end{gather*}
For the proof of these two facts {see e.g.} \cite[Lemma~II.5 ]{LinRiv2} and \cite[Lemma~8]{majumdarzarnescu10}. Using this in \eqref{monot1} we obtain
\begin{equation}\label{monot2}
\frac{d\varphi}{d\rho}\geq -C(\Omega,u_b)\left(1 +  \rho^{3-n}\int_{\partial\Omega\cap B_\rho}\abs{\frac{\partial u}{\partial\nu}}^2 \right),
\end{equation}
for some constant $C(\Omega,u_b)>0$.
For $n=3$, one may conclude using \cite[Lemma~10]{majumdarzarnescu10}. But we want to deal with general $n\geq 3$, and from this point on our proof departs from \cite{majumdarzarnescu10}. Consider a smooth function $\chi(r)$ satisfying
\begin{equation}\label{chi}
\abs{\chi}\leq 1,\quad\abs{\chi'}\leq 2,\quad\chi\equiv 1\text{ in }[0,1],\quad \chi\equiv 0\text{ in }[2,\infty),
\end{equation}
and let $\chi_\rho(x):=\chi(\abs{x}/\rho)$. Fix also a smooth vector field $X$ such that $X=\nu$ on $\partial\Omega$. Multiplying \eqref{statio} by $\chi_\rho X$ and integrating in $\Omega$ we have
\begin{align*}
\frac 12 \int_{\partial\Omega\cap B_\rho}\abs{\frac{\partial u}{\partial\nu}}^2 &
\leq \frac 12 \int_{\partial\Omega}\chi_\rho \abs{\frac{\partial u}{\partial\nu}}^2 \\
& = \frac 12 \int_{\partial\Omega}\chi_\rho \abs{\nabla_{\partial\Omega}u_b}^2 + \int_{\Omega}\chi_\rho (\partial_l X_j) T_{\ell j} +\int_{\Omega} \partial_\ell\chi_\rho X_j T_{\ell j}\\
&\leq C(\Omega, u_b)\left( \rho^{n-1} + \rho^{n-2}\varphi(2\rho) + \rho^{n-3}\varphi(2\rho) \right).
\end{align*}
Plugging this estimate into \eqref{monot2} we find
\begin{equation*}
\frac{d\varphi}{d\rho}\geq -C(\Omega,u_b)\left(1 +  \varphi(2\rho) \right),
\end{equation*}
which, recalling $\psi(\rho)=2K\rho+\varphi(\rho)$,  gives \eqref{monot} for $K=C$.
\end{proof}

The relevance of the monotonicity formula provided by Proposition~\ref{prop:monot} is that it allows to deduce smallness of the energy at all scales from smallness of the energy at one fixed scale, in the following sense:

\begin{lem}\label{lem:quasimonot}
There exist $\rho_*>0$ and $\alpha_0>0$ depending on $\Omega$ and $u_b$ such that for any $x_0\in\overline\Omega$ and any $\rho_0\in (0,\rho_*)$, if
\begin{equation*}
\rho^{2-n}\int_{\Omega\cap B_\rho(x_0)}e_\e(u_\e) \leq \alpha \leq \alpha_0 \qquad \forall \rho\in [\rho_0,2\rho_0],
\end{equation*}
then
\begin{equation*}
\rho^{2-n}\int_{\Omega\cap B_\rho(x_0)}e_\e(u_\e) \leq  \alpha + 2K \rho_0 \qquad \forall \rho\in (0,\rho_0).
\end{equation*}
\end{lem}
\begin{proof}
If $\alpha_0$ and $\rho_*$ are small enough (depending on $K$) then $\psi(\rho)\leq 1/2$ for all $\rho\in [\rho_0,2\rho_0]$. Let
\begin{equation*}
\rho_1:=\inf\left\lbrace \rho\in [0,\rho_0]\colon \psi(r)\leq 1/2 \;\; \forall r\in (\rho,\rho_0]\right\rbrace,
\end{equation*}
and assume that $\rho_1>0$. Then it would hold $\psi(2\rho_1)\leq1/2$ and by \eqref{monot} this implies $\psi'(\rho_1)>0$, so that $\psi(r)<\psi(\rho_1)\leq 1/2$ for all $r \in (\rho_1-\delta,\rho_1)$, contradicting the definition of $\rho_1$. We deduce that $\psi\leq 1/2$ and $d\psi/d\rho>0$ in $(0,\rho_0]$. Therefore it holds $\psi\leq \psi(\rho_0)$ and this concludes the proof.
\end{proof}

\section{Small energy estimate}\label{s:estim}

In this section we derive the small energy estimate that provides a uniform Lipschitz bound provided the energy is small at all scales.

\begin{prop}\label{prop:smallestim}
There exist $\e_0>0$, $\eta_0>0$ and $C>0$ (depending on {$f$, $\Omega$ and $u_b$}) such that for all $\e\in (0,\e_0)$, {$r\in (0,1)$} and $x_0\in\overline\Omega$, any smooth solution $u$ of \eqref{pde}-{\eqref{ub}} with
\begin{equation*}
E:=\sup_{B_\rho(x)\subset B_{2r}(x_0)} \; \rho^{2-n}\int_{\Omega\cap B_\rho(x)} e_\e(u) \leq \eta_0,
\end{equation*}
satisfies
\begin{equation*}
r^{2}\sup_{B_{r/2}}\, e_\e(u) \leq C\, \left( E + {r^2} \right).
\end{equation*}
\end{prop}

The strategy of the proof is the same as in \cite[Lemma~12]{nguyenzarnescu13}.
One crucial ingredient is a Bochner-type inequality which provides an elliptic equation satisfied by the energy density $e_\e(u)$:
\begin{lem}\label{lem:bochner}
There exists $\delta>0$ and $C>0$ depending only on the potential $f$ such that for any smooth solution $u$ of \eqref{pde} it holds
\begin{equation}\label{bochner}
-\Delta [e_\e(u)] \leq C e_\e(u)^2\quad\text{at }x\in\Omega,
\end{equation}
provided $\dist(u(x),\mathcal N)< \delta$.
\end{lem}
In \cite{majumdarzarnescu10} the proof is provided in the special case of the liquid crystal potential. In the general case there is no additional difficulty. We present the proof here in order to make transparent how the only crucial assumption is the nondegeneracy \eqref{nondegen}. First we set a bit of notation and reformulate \eqref{nondegen} into the form that we are actually going to use.

For $\delta>0$ we denote by $\mathcal N_\delta$ the tubular $\delta$-neighborhood of $\mathcal N$,
\begin{equation*}
\mathcal N_\delta :=\left\lbrace z\in\R^k \colon \dist(z,\mathcal N)< \delta\right\rbrace.
\end{equation*}
There exists $\delta>0$ such that the canonical projection
\begin{equation*}
\pi=\pi_{\mathcal N}\colon \mathcal N_\delta  \longrightarrow \mathcal N\subset\R^k,
\end{equation*}
is well-defined and smooth. Note that the differential of $\pi$ at $z\in\mathcal N_\delta$ is simply the orthogonal projection on $T_{\pi(z)}\mathcal N$:
\begin{equation*}
D\pi(z)=\pi_{tan}(z):=\proj_{T_{\pi(z)}\mathcal N}\in\mathcal L(\R^k).
\end{equation*}
We denote by $\pi_{nor}(z)$ the projection on $(T_{\pi(z)}\mathcal N)^\perp$,
\begin{equation*}
\pi_{nor}(z):=I-\pi_{tan}(z)=\proj_{(T_{\pi(z)}\mathcal N)^\perp}\in\mathcal L(\R^k).
\end{equation*}
Next, we write the potential in a form that adapts well to our purposes in that it really emphasizes how it all boils down to nondegeneracy.  To that end let us observe that a Taylor expansion, for $z\in\mathcal N_\delta ,$ yields
\begin{equation}\label{A}
f(z)=z^\perp \cdot A(z) z^\perp,\quad z^\perp:=z-\pi(z),
\end{equation}
for some smooth map $A\colon \mathcal N_\delta \to \R^{k\times k}_{sym}$. More precisely, the representation \eqref{A} follows from   Taylor's formula for the function $t\mapsto f(t z +(1-t)\pi(z))$ between $t=0$ and $t=1$, using the facts that $f(\pi(z))=0$ and $\nabla f(\pi(z))=0$, and the map $A$ can be explicitly expressed as
\begin{equation*}
A(z)=\int_0^1 (1-t) \nabla^2 f(t z +(1-t)\pi(z)) \,dt.
\end{equation*}
Let us also notice that \eqref{N} and the nondegeneracy assumption \eqref{nondegen} ensure that, provided $\delta$ is small enough, $A(z)$ is uniformly positive definite in the direction normal to $\mathcal N$, that is
\begin{equation}\label{csqnondegen}
\xi\cdot A(z)\xi\geq \alpha_0 \abs{\xi}^2\quad\forall \xi \perp T_{\pi(z)}\mathcal N,
\end{equation}
for some $\alpha_0>0$.
We may now proceed to the proof of the Bochner inequality.

\begin{proof}[Proof of Lemma~\ref{lem:bochner}.]
We write $e=e_\e(u)$ and compute
\begin{equation}\label{bochner1}
\Delta e =\abs{\nabla^2 u}^2 +\abs{\Delta u}^2 +\frac {2}{\e^2}\partial_k u \cdot (\nabla^2 f(u) \partial_k u).
\end{equation}
From \eqref{A}, we see that
\begin{align}
\nabla f(z)& = 2\pi_{nor}(z)A(z)z^\perp + z^\perp \cdot\nabla A(z) z^\perp,\label{nablaf}\\
\nabla^2 f(z) & =2\pi_{nor}(z)A(z)\pi_{nor}(z) +z^\perp\cdot\nabla^2 A(z) z^\perp \nonumber \\
&\quad  +{4}\pi_{nor}(z) \nabla A(z) z^\perp  + 2\nabla\pi_{nor}(z) A(z) z^\perp .\label{hessf}
\end{align}
{The first term in the expression of $\nabla^2 f(z)$ being a nonnegative symmetric matrix thanks to \eqref{csqnondegen},} implies that for any $\xi\in\R^k$ and $z\in\mathcal N_\delta$ we have
\begin{equation*}
\xi\cdot\nabla^2 f(z)\xi \geq  -C\abs{\xi}^2\abs{z^\perp},
\end{equation*}
from which we infer
\begin{equation}\label{bochner2}
\partial_k u \cdot (\nabla^2 f(u) \partial_k u)\geq -C \left( \frac{\eta}{\e^2} \abs{u^\perp}^2 + \frac{\e^2}{\eta} \abs{\nabla u}^4\right),
\end{equation}
for an arbitrary $\eta>0$, to be chosen later. Plugging \eqref{bochner2} into \eqref{bochner1} we deduce that
\begin{equation*}
-\Delta e \leq \frac{C}{\eta}\abs{\nabla u}^4 + \left( C\frac{\eta}{\e^4}\abs{u^\perp}^2 - \abs{\Delta u}^2\right).
\end{equation*}
Finally we remark that, provided $\delta$ is chosen small enough, \eqref{csqnondegen}-\eqref{nablaf} ensure
\begin{equation*}
\abs{\Delta u}^2=\frac{1}{\e^4}\abs{\nabla f(u)}^2 \geq \frac{1}{\e^4}\frac{\alpha_0^2}{2} \abs{u^\perp}^2\quad\text{ if }\dist(u,\mathcal N)<\delta.
\end{equation*}
Choose $\eta\leq\alpha_0^2/2C$ to finish the proof.
\end{proof}

As explained in the introduction, the main point at which our proof of Proposition~\ref{prop:smallestim} differs from \cite{nguyenzarnescu13} is the treatment of the estimates for $\abs{\nabla u}$ on $\partial\Omega$ that are needed to make good use of the Bochner inequality at the boundary. While in \cite{nguyenzarnescu13} the authors relied heavily on the particular structure of their potential, our argument, closer to \cite{chenlin93}, uses only the nondegeneracy assumption \eqref{nondegen}.

\begin{lem}\label{lem:bdryestim} Let $p>n .$
There exist $\delta,C>0$ (depending on $p, f$ and $\Omega$) such that for any $x_0\in\overline\Omega$, $r,\e\in(0,1]$, and $u$ smooth solution of \eqref{pde} with boundary conditions \eqref{ub}, if
\begin{equation*}
\dist(u,\mathcal N)\leq\delta \mbox{ in } B_r(x)\cap\Omega,
\end{equation*}
then
it holds
\begin{eqnarray*}
\sup_{B_{r/2}(x)\cap \partial\Omega} \abs{\nabla u} \leq C \left( r^{{1-\frac np}}\norm{e_{{\e/r}}(u)}_{L^p(B_r(x)\cap\Omega)}  \right.
&+& r^{-n/2}\norm{\nabla u}_{L^2(B_r\cap\Omega)} \\
&+&\left. \sup_{\partial\Omega}|\nabla u|+\sup_{\partial\Omega} r|\nabla^2 u|+\frac\delta r \right).
\end{eqnarray*}
\end{lem}

\begin{proof}  First note that if $x_0 \in \Omega ,$ the conclusion of the Lemma is vacuously true for sufficiently small radii $r$. Thus, the proof really requires care for $r\geq 1/2 \dist(x_0,\partial \Omega).$
As usual, it suffices to prove the estimate for $r=1$, {the general case following by rescaling.}
Provided $\delta$ is small enough, we may write $u=u^\mathcal N +u^\perp$ where $u^\mathcal N :=\pi_\mathcal N (u)$ is smooth. Using the fact that for a smooth map $v$ with values into $\mathcal N\subset\R^k$, the normal component of its Laplacian is given by
\begin{equation*}
\pi_{nor}(v)\Delta v = II_\mathcal N(v)[\nabla v, \nabla v],
\end{equation*}
 where $II_\mathcal N$ denotes the second fundamental form of $\mathcal N$, we compute the equation satisfied by $u^\mathcal N$:
\begin{equation}\label{eqaN1}
\begin{aligned}
\Delta  u^\mathcal N & = \pi_{nor}(u^{\mathcal N})\Delta u^{\mathcal N}+ \pi_{tan}(u^\mathcal N)\Delta u^{\mathcal N} \\
& =II_\mathcal N(u^\mathcal N)[\nabla u^\mathcal N,\nabla u^\mathcal N] + \pi_{tan}(u^\mathcal N)\Delta u -\pi_{tan}(u^\mathcal N)\Delta u^\perp\\
& = II_\mathcal N(u^\mathcal N)[\nabla u^\mathcal N,\nabla u^\mathcal N] + \frac{1}{\e^2}\pi_{tan}(u^\mathcal N)\nabla f(u)-\pi_{tan}(u^\mathcal N)\Delta u^\perp.
\end{aligned}
\end{equation}
For the last equality we used \eqref{pde}. To compute the last term we remark that taking the Laplacian of the identity $\pi_{tan}(u^\mathcal N)u^\perp\equiv 0$ yields
\begin{align*}
-\pi_{tan}(u^\mathcal N)\Delta u^\perp & =
2\nabla \pi_{tan}(u^\mathcal{N})\cdot \nabla u^\perp +\Delta[\pi_{tan}(u^\mathcal N)]u^\perp \\
& = 2\nabla \pi_{tan}(u^\mathcal{N})\cdot \nabla u^\perp + \nabla \pi_{tan}(u^\mathcal N)u^\perp\cdot \Delta u^\mathcal N\\
&\quad + \nabla^2\pi_{tan}(u^\mathcal N)\left[ \nabla u^\mathcal N,\nabla u^\mathcal N \right]u^\perp.
\end{align*}
Plugging this into \eqref{eqaN1} and recalling \eqref{nablaf} we obtain
\begin{align*}
\Delta u^\mathcal N & = II_\mathcal N (u^\mathcal N)[\nabla u^\mathcal N,\nabla u^\mathcal N] +\frac{1}{\e^2}\pi_{tan}(u^\mathcal N) (u^\perp\cdot \nabla A(u)u^\perp)\\
& \quad + \nabla\pi_{tan}(u^\mathcal N)u^\perp \cdot \Delta u^\mathcal N +\nabla^2\pi_{tan}(u^\mathcal N)[\nabla u^\mathcal N,\nabla u^\mathcal N]u^\perp \\
&\quad  + 2 \nabla [ \pi_{tan}(u^\mathcal N)]\cdot\nabla u^\perp.
\end{align*}
Since $\abs{u^\perp}^2\leq C f(u)$ and $\abs{u^\perp}\leq\delta$ we deduce that it holds
\begin{equation*}
\abs{\Delta u^\mathcal N}\leq C \left(\delta\abs{\Delta u^\mathcal N} +e_\e(u)\right),
\end{equation*}
for a constant $C>0$ depending on $\mathcal N$ and $f$. Choosing $\delta$ small enough we find
\begin{equation}\label{DeltauN}
\abs{\Delta u^\mathcal N}\leq C e_\e(u).
\end{equation}
Then elliptic estimates as in  \cite[Lemma 11]{nguyenzarnescu13} yield
\begin{equation}\label{estimgraduN}
\sup_{B_{1/2}\cap\Omega}\abs{\nabla u^\mathcal N} \leq C\left(\norm{e_\e(u)}_{L^p(B_1\cap\Omega)} + \norm{\nabla u}_{L^2(B_1\cap\Omega)}+ \norm{u_b}_{C^2(\Omega)} \right).
\end{equation}
It remains to bound $\nabla u^\perp$. Since $u_b$ takes values into $\mathcal N$ we have $u^\perp_{\lfloor\partial\Omega}=0$ and it suffices to estimate the normal derivative. First we note that $\abs{u^\perp}\Delta\abs{u^\perp}  \geq u^\perp\cdot \Delta u^\perp$, as can be seen for instance from the identities
\begin{align*}
 2u^\perp\cdot \Delta u^\perp+|\nabla u^\perp|^2=\Delta (|u^\perp|^2)= 2|u^\perp|\Delta |u^\perp|+ |\nabla |u^\perp||^2,
\end{align*}
together with the inequality $|\nabla |u^\perp||^2 \leq |\nabla u^\perp|^2$ {which follows from $\p_i |u^\perp|=u^\perp \cdot\partial_i u^\perp / \abs{u^\perp}$}. Then we use \eqref{pde} and \eqref{nablaf}, to calculate
\begin{align*}
\abs{u^\perp}\Delta\abs{u^\perp} & \geq u^\perp\cdot \Delta u^\perp \\
& = u^\perp\cdot \Delta u - u^\perp\cdot \Delta u^\mathcal N\\
&{= \frac{1}{\e^2} u^\perp\cdot \nabla f(u) - u^\perp\cdot \Delta u^\mathcal N}\\
& = \frac{2}{\e^2}(u^\perp\cdot A(u)u^\perp) + \frac{1}{\e^2}u^\perp\cdot(u^\perp\cdot \nabla A(u) u^\perp) - u^\perp\cdot \Delta u^\mathcal N\\
&\geq -C \abs{u^\perp}e_\e(u).
\end{align*}
For the last inequality we used \eqref{DeltauN} and the facts, implied by \eqref{nondegen}, that $\abs{u^\perp}^2\leq C f(u)$
and $u^\perp\cdot A(u)u^\perp\geq 0$.
Therefore we have
{
\begin{equation}\label{Deltauperp}
-\Delta \abs{u^\perp}\leq C e_\e(u),
\end{equation}
}
and by the maximum principle it holds $\abs{u^\perp}\leq w$, where
\begin{equation*}
-\Delta w = Ce_\e(u) \text{ in }B_1\cap\Omega,\qquad
w=\abs{u^\perp}  \text{ on }\partial(B_1\cap\Omega).
\end{equation*}
{Since $\abs{u^\perp}=0$ on $B_1\cap\partial\Omega$, elliptic estimates as in  \cite[Lemma~11]{nguyenzarnescu13} imply
\begin{equation*}
\sup_{B_{1/2}\cap\Omega}\abs{\nabla w}\leq C \left( \norm{e_\e(u)}_{L^p(B_{3/4}\cap\Omega)} + \norm{\nabla w}_{L^2(B_{3/4}\cap\Omega)}\right).
\end{equation*}
{To estimate the last term one may proceed as in \cite[Lemma~9]{nguyenzarnescu13}. We only sketch the argument here:} splitting $w$ as $w=w_1+w_2$, where $\Delta w_2=0$ and $w_1$ vanishes on the full boundary $\partial(B_1\cap\Omega)$, we have the estimates
\begin{align*}
\norm{w_1}_{L^2(B_{3/4}\cap\Omega)}& \leq C\norm{e_\e(u)}_{L^2(B_1\cap\Omega)}\leq C \norm{e_\e(u)}_{L^p(B_1\cap\Omega)},\\
\norm{w_2}_{L^2(B_{3/4}\cap\Omega)}&\leq C \norm{w_2}_{L^\infty(\partial(B_1\cap\Omega))}\leq C\delta.
\end{align*}
We deduce
\begin{equation*}
\sup_{B_{1/2}\cap\Omega}\abs{\nabla w}\leq C \left( \norm{e_\e(u)}_{L^p(B_{1}\cap\Omega)} + \delta \right).
\end{equation*}
In particular we have the inequalities
\begin{equation*}
\abs{u^\perp}\leq w \leq C \left( \norm{e_\e(u)}_{L^p(B_1\cap\Omega)}+\delta\right) \dist(\cdot,\partial\Omega),
\end{equation*}
which imply
\begin{equation}\label{estimgraduperp}
\sup_{B_{1/2}\cap\partial\Omega}\abs{\frac{\partial u^\perp}{\partial\nu}}\leq C \left(\norm{e_\e(u)}_{L^p(B_1\cap\Omega)}+\delta\right).
\end{equation}
}
The conclusion follows from \eqref{estimgraduN}-\eqref{estimgraduperp}.
\end{proof}

{
\begin{rem}\label{rem:maindiff}
With respect to \cite{nguyenzarnescu13}, our above treatment of the estimate for $\abs{\nabla u}$ on the boundary $\partial\Omega$ is simplified and works for general nonlinearities because we are able to derive the differential inequality \eqref{Deltauperp} satisfied by $\abs{u^\perp}=\dist(u,\mathcal N)$.
\end{rem}
}

Equipped with Lemma~\ref{lem:bochner} and Lemma~\ref{lem:bdryestim} we may now proceed to the proof of the small energy estimate, following \cite[Lemma~12]{nguyenzarnescu13} quite closely. We provide the details of the argument in our setting in the lines below.

\begin{proof}[Proof of Proposition~\ref{prop:smallestim}.]

\textbf{Step 1.} \textit{Rescaling.}

We use coordinates in which $x_0=0$.
We will show that
\begin{equation}\label{smallestim1}
M:=\sup_{0<\rho<r} (r-\rho)^2 \sup_{B_\rho\cap\Omega} \left( e_\e (u) - L \right) \leq C\,E,
\end{equation}
for some $C,L>0$ to be chosen,
which implies the conclusion.
This allows to make use of a rescaling trick introduced in \cite{schoen84} in the context of harmonic maps. There exist $\rho_0\in [0,r]$ and $x_1\in \overline B_{\rho_0}\cap\overline\Omega$ such that
\begin{equation*}
M=(r-\rho_0)^2 \sup_{B_{\rho_0}\cap\Omega} \left( e_\e(u) -{L}\right) = (r-\rho_0)^2 \left[ e_\e(u)(x_1) -{L}\right].
\end{equation*}
With $\rho_1:=(r-\rho_0)/2$, it holds $M=4\rho_1^2 [ e_\e(u)(x_1)-{L} ]$ and
\begin{align*}
\sup_{B_{\rho_1}(x_1)\cap\Omega}e_\e(u) & \leq \sup_{B_{\rho_1+\rho_0}\cap\Omega}e_\e(u) \leq \frac{M}{(r-\rho_1-\rho_0)^2}+{L} \\
& =\frac{M}{\rho_1^2}+{L}=4 e_\e(u)(x_1).
\end{align*}
Therefore, setting $V:=e_\e(u)(x_1)$, $\rho_2:=\rho_1\sqrt V$, $\widetilde\Omega:=\sqrt V(\Omega-x_1)$ and
\begin{equation*}
v(x)=\frac{1}{V}e_\e(u)(x_1+V^{-1/2}x)\quad\text{for }x\in B_{\rho_2}\cap\widetilde\Omega,\ \left(x_1+V^{-\frac12}x \in B_{\rho_1}(x_1) \right)
\end{equation*}
we find that it holds
\begin{gather}
1=v(0)\leq \sup_{B_{\rho_2}\cap\widetilde\Omega} v \leq 4,\label{estimv1}\\
\rho^{2-n}\int_{B_\rho\cap\widetilde\Omega} v \leq E\leq\eta_0\quad\text{for all }\rho\leq\rho_2.\label{estimv2}
\end{gather}
Note that we may assume
\begin{equation}\label{VLub}
V\geq {L},
\end{equation}
because otherwise $M=4\rho_1^2(V-{L})\leq 0$ and \eqref{smallestim1} is trivial.

\textbf{Step 2.} \textit{It holds $\rho_2 \leq 1$ (provided $\eta_0$ and $1/L$ are small enough).}

Assume that $\rho_2>1$. Let
\begin{equation*}
\widetilde u(x):=u(x_1+V^{-1/2}x)\quad\text{for }x\in B_{\rho_2}\cap\widetilde\Omega.
\end{equation*}
It holds
\begin{equation}\label{ftildeu}
\frac{1}{\e^2 V} f(\widetilde u)\leq v\leq 4\quad\text{in }B_{\rho_2}\cap\widetilde\Omega.
\end{equation}
We would like to deduce that $f(\widetilde u)$ is small, which will imply that $\dist(u,\mathcal N)$ is small thanks to the nondegeneracy assumption \eqref{nondegen}, and therefore allow us to use the Bochner-type inequality \eqref{bochner}.

Since $\Delta \widetilde u=\frac{1}{\e^2 V}\nabla f(\widetilde u)$, rescaled elliptic estimates \cite[Lemma~11]{nguyenzarnescu13} yield
\begin{equation}\label{smallestim2}
\sup_{B_{1/2}\cap\widetilde\Omega}\abs{\nabla\widetilde u}\leq C\left(\frac{1}{\e^2V}\norm{\nabla f(\widetilde u)}_{L^p(B_1\cap\widetilde\Omega)} +{V^{-1/2}}+{\norm{\nabla \widetilde u}_{L^2(B_1\cap\widetilde\Omega)}}\right).
\end{equation}

Recall that thanks to the nondegeneracy assumption \eqref{nondegen} we have
\begin{equation*}
\alpha_0 \abs{z^\perp}^2 \leq z^\perp\cdot A(z)z^\perp =f(z),
\end{equation*}
for all $z$ close enough to $\mathcal N$.
Using this and the expression \eqref{nablaf} for $\nabla f$, together with the uniform bound $\abs{\widetilde u}\leq R + \sup \abs{u_b}$ \eqref{unifbound} we find that
\begin{equation*}
\abs{\nabla f(\widetilde u)}\leq C\abs{\widetilde u^\perp} \leq C\sqrt{f(\widetilde u)}.
\end{equation*}
Hence it holds
\begin{equation*}
\abs{\nabla f(\widetilde u)}^p\leq C f(\widetilde u)^{p/2}\leq C f(\widetilde u),
\end{equation*}
for some constant $C>0$ depending on $\mathcal N$, $f$ and $p>2$. Therefore, using \eqref{ftildeu} and \eqref{estimv2} we find
\begin{equation*}
\norm{\nabla f(\widetilde u)}_{L^p(B_1\cap\widetilde\Omega)}\leq C (\e^2 V \eta_0)^{1/p}.
\end{equation*}
Plugging this and \eqref{VLub} into \eqref{smallestim2} we have
\begin{equation*}
\sup_{B_{1/2}\cap\widetilde\Omega}\abs{\nabla\widetilde u}\leq C\left(\frac{\eta_0^{1/p}}{(\e^2 V)^{1-1/p}}+\xi\right),\qquad
{\xi:=\frac{1}{L^{1/2}}+\eta_0^{1/2}}.
\end{equation*}
From the mean value theorem,
the smoothness of $f$ and the uniform bound \eqref{unifbound} we have
\begin{equation*}
f(\widetilde u)(x)\leq f(\widetilde u)(y) +C \sup_{B_{1/2}}\abs{\nabla \widetilde u}\quad \forall x,y\in B_{1/2}\cap\widetilde\Omega,
\end{equation*}
Integrating this inequality over $y\in B_{1/2}\cap\widetilde\Omega$ and using again \eqref{estimv2} one finds that for $x \in B_{1/2}\cap\widetilde\Omega$ it holds
\begin{equation*}
\frac{1}{\e^2 V} f(\widetilde u) \leq {\eta_0} +C\left(\frac{\eta_0^{1/p}}{(\e^2V)^{2-\frac{1}{p}}}+\frac{\xi}{\e^2 V}\right),
\end{equation*}
and therefore
\begin{align*}
\sup_{B_{1/2}\cap\widetilde\Omega} v &=\sup_{B_{1/2}\cap\widetilde\Omega} \frac 12 \abs{\nabla \widetilde u}^2 +\frac{1}{\e^2V}f(\widetilde u)  \\
&\leq  C\left[\left( \frac{\eta_0^{1/p}}{(\e^2V)^{1-\frac {1}{p}}} +\xi \right)^2+ \eta_0 +\frac{\eta_0^{1/p}}{(\e^2V)^{2-\frac1p}}+\frac{\xi}{\e^2V}\right].
\end{align*}
Recalling \eqref{estimv1} we deduce
\begin{equation*}
1\leq \sup_{B_{1/2}\cap\widetilde\Omega} v \leq
C\left(
\frac{\eta_0^{2/p}}{(\e^2V)^{2-\frac {2}{p}}}
+\xi^2 + \eta_0
+\frac{\eta_0^{1/p}}{(\e^2V)^{2-\frac1p}}+\frac{\xi}{\e^2V}
\right).
\end{equation*}
Since $\xi$ is arbitrarily small for small enough $\eta_0$ and $1/L$, we infer that given any $\delta_0>0$ it must hold $\e^2 V\leq \delta_0$, provided $\eta_0$ and $1/L$ are small enough. Recalling \eqref{ftildeu} we may therefore choose $\eta_0$ and $L$ in such a way that $\dist(\widetilde u,\mathcal N)<\delta$ and the Bochner-type inequality \eqref{bochner} holds for $\widetilde u$.
This implies
\begin{equation}\label{bochnerv}
-\Delta v \leq C v^2 \leq 4C\, v \quad\text{in }B_{1}\cap\widetilde\Omega.
\end{equation}
On the other hand, since on $\partial\Omega$ it holds $v=\abs{\nabla u}^2/2$ we deduce from Lemma~\ref{lem:bdryestim} the estimate
\begin{align*}
\sup_{B_{1/2}\cap\partial\widetilde\Omega} v & \leq C \left( {(1+V^{-1})}\norm{v}_{L^p(B_1)}^2 +{V^{-1}} +{\norm{v}_{L^1(B_1\cap\widetilde\Omega)}^{1/2} +\delta }\right)\\
&\leq C \left( \eta_0^{1/p} +\frac 1 {L}{ + \delta}\right).
\end{align*}
For the last inequality we used \eqref{estimv2} and \eqref{VLub}. We choose $\eta_0$ and $1/L$ small enough to ensure that $\delta$ is small and that $v\leq 1/4$ on $B_{1/2}\cap\partial\widetilde\Omega$.  Then the function
\begin{equation*}
\widetilde v:=\begin{cases}
\max(v-1/2,0)&\text{ in }B_{1/2}\cap\widetilde\Omega,\\
0 &\text{ in }B_{1/2}\setminus\widetilde\Omega,
\end{cases}
\end{equation*}
satisfies $-\Delta \widetilde v \leq C\, \widetilde v$ in $B_{1/2}$ and Harnack's inequality yields
\begin{equation*}
1/2 \leq \widetilde v(0)\leq c\int_{B_{1/2}}\widetilde v\leq c\eta_0,
\end{equation*}
which implies a contradiction provided $\eta_0$ is small enough. This proves $\rho_2\leq 1$.

\textbf{Step 3.} \textit{We conclude (provided $\e_0$ is small enough).}

Since $\rho_2\leq 1$ it holds $M=4\rho_1^2[V-L]\leq 4\rho_2^2\leq 4$, and in particular $\e^{-2}f(u)\leq e_\e(u)\leq C$ in $B_{r/2}$. In fact the same argument applies in any ball $B_r(\widetilde x_0)$ with $\abs{\widetilde x_0 -x_0}<r$, so that we have $\e^{-2}f(u)\leq C$ in $B_r$. Therefore the Bochner inequality \eqref{bochner} holds in $B_r$ provided $\e_0$ is small enough, and we deduce (using \eqref{estimv1} as above)
\begin{equation*}
-\Delta v \leq C v^2 \leq 4C\, v \quad\text{in }B_{\rho_2}\cap\widetilde\Omega.
\end{equation*}
As in Step~2, Lemma~\ref{lem:bdryestim} ensures
\begin{align*}
\sup_{B_{\rho_2/2}\cap\partial\widetilde\Omega} v \leq C\left( \eta_0^{1/p} +\frac {1} {L} {+ \e^2}\right) \leq 1/4,
\end{align*}
and we consider the function
\begin{equation*}
\widetilde v:=\begin{cases}
\max(v-1/2,0)&\text{ in }B_{\rho_2/2}\cap\widetilde\Omega,\\
0 &\text{ in }B_{\rho_2/2}\setminus\widetilde\Omega,
\end{cases}
\end{equation*}
which satisfies $-\Delta \widetilde v \leq C\, \widetilde v$ in $B_{\rho_2/2}$.
Letting now
\begin{equation*}
w(x):=\widetilde v(\rho_2 x)\quad \abs{x}<1,
\end{equation*}
it holds
$-\Delta w \leq 4C \rho^2_2\, w \leq 4C w$, and
Harnack's inequality yields
\begin{equation*}
1/2   \leq w(0)\leq c\int_{B_1}w = c\rho_2^{-n} \int_{B_{\rho_2}}\widetilde v
 \leq C \rho_2^{-2} E,
\end{equation*}
hence $8CE\geq 4\rho_2^2 =4\rho_1^2 V \geq M$, which concludes the proof.
\end{proof}

\section{Proof of Theorem~\ref{thm:main}}\label{s:proofmain}

Consider a compact $X\subset\overline\Omega\setminus \sing(u_\star)$. Let $\rho_*$, $\alpha_0$ and $K$ be as in Lemma~\ref{lem:quasimonot} and $\eta_0$ be as in the small energy estimate Proposition~\ref{prop:smallestim}.
Since $u_\star$ is smooth in a compact neighborhood of $X$, we may choose $\rho_0\in (0,\rho_*)$ with $4K\rho_0\leq\eta_0$ such that for all $x_0\in X$ we have
\begin{equation*}
\rho^{2-n}\int_{\Omega\cap B_\rho(x_0)}\abs{\nabla u_\star}^2 \leq \frac 14 \min(\eta_0,\alpha_0) \qquad \forall \rho\in [\rho_0,2\rho_0].
\end{equation*}
Then, for $\rho\in [\rho_0,2\rho_0]$ we find
\begin{align*}
\rho^{2-n}\int_{\Omega\cap B_\rho(x_0)}e_\e(u_\e) & \leq \frac 14 \min(\eta_0,\alpha_0) + \rho_0^{2-n}\left( \int_\Omega\abs{\nabla u_\e -\nabla u_*}^2 +  \frac{1}{\e^2}\int_\Omega f(u_\e) \right).
\end{align*}
Since $u_\e\to u_\star$ in $H^1$ and the minimality of $u_\e$ implies $\e^{-2}\int_\Omega f(u_\e)\to 0$ (by comparing with $u_\star$) we may choose $\epsilon_1\in (0,\e_0)$ (with $\e_0$ as in Proposition~\ref{prop:smallestim}) such that for all $\e\in (0,\e_1)$ it holds
\begin{equation*}
\rho^{2-n}\int_{\Omega\cap B_\rho(x_0)}e_\e(u_\e) \leq \frac 12 \min(\eta_0,\alpha_0) \qquad \forall \rho\in [\rho_0,2\rho_0].
\end{equation*}
By Lemma~\ref{lem:quasimonot} we deduce
\begin{equation*}
\rho^{2-n}\int_{\Omega\cap B_\rho(x_0)}e_\e(u_\e) \leq \eta_0,
\end{equation*}
for all $x_0\in X$ and all $\rho\in (0,\rho_0)$. This allows us to apply Proposition~\ref{prop:smallestim} to conclude that
\begin{equation*}
\sup_X \abs{\nabla u_\e} \leq C(X),
\end{equation*}
so that by Arzela-Ascoli's theorem, $u_\e$ converges in fact uniformly in $X$.\qed

\

\textbf{Acknowledgements:} R.R has been supported by the Millennium Nucleus Center for Analysis of PDE NC130017 of the Chilean Ministry of Economy.

\bibliographystyle{plain}
\bibliography{convQ}

\end{document}